\documentclass{article}
\usepackage{graphicx,amsmath,amsfonts,amssymb,amsthm}
\usepackage[dvipsnames,svgnames]{xcolor} 
\usepackage[margin=1in]{geometry}% can also add 'nohead', or leave blank

\newtheorem{corollary}{Corollary}
\newtheorem{proposition}{Proposition}

\usepackage{listings}
\usepackage{fancyhdr}
\newcommand{\helv}{%
    \fontfamily{phv}\fontsize{9}{11}\selectfont}
\usepackage{hyperref} 
\hypersetup{
  urlcolor=black,
  colorlinks=true,
  citecolor=MidnightBlue,
  urlcolor=Bittersweet,
  pdftitle={Singular value soft-thresholding via the polar decomposition},
  pdfauthor={Stephen Becker}
}
\newcommand{\defeq}{\stackrel{\text{\tiny def}}{=}}
\DeclareMathOperator{\polar}{polar}

\title{Singular value soft-thresholding via the polar decomposition}

\author{Stephen Becker\thanks{\texttt{stephen.becker@colorado.edu}, University of Colorado Boulder, USA}}

\date{\today}

\begin{document}

\maketitle
\begin{abstract}
Singular value soft-thresholding can be computed via a reduction to the matrix polar decomposition, which allows one to exploit GPU-friendly algorithms for computing the polar decomposition.  Empirically, there is a significant speed-up on  GPUs compared to the standard approach using the SVD. We leave the investigation of robustness to future work, but note that due to the discontinuous nature of the sign function, the reduction to the polar decomposition is likely only suitable for low-accuracy applications.
\end{abstract}

\thispagestyle{fancy} 

\section{Motivation}
The polar decomposition can be computed via the SVD, but recently it has been popular to compute it to low-accuracy via communication-friendly algorithms such as the Newton-Schulz iteration (see, e.g., \cite{higham2008functions}); this is done as a step in some transformer-based neural net training methods such as Muon. The dominant cost of Newton-Schulz based methods is almost exclusively in matrix-matrix multiplications, which is communication-friendly and fast on a GPU. Quoting \cite{chen2014stable}, \begin{quote}
``The best communication cost of matrix-matrix multiplications is $\log p$ times smaller than that of LU and QR, where $p$ is the number of processors''    
\end{quote}
and presumably also has better communication cost compared to eigenvalue and singular value decompositions.

We show that polar decompositions can be used to compute singular value soft-thresholding, and therefore provide a low-communication GPU-friendly algorithm. There is already some work \cite{SVD_NewtonSchulz} showing that polar decompositions can be used to compute the SVD, and hence could do singular value soft-thresholding, but this still relies on an eigendecomposition and QR and is thus not as efficient as we would like. Our idea is to avoid computing the SVD altogether.

\section{The algorithm}
If $A=U\Sigma V^\top$ is the SVD of a $m\times n$ matrix $A$, then the polar decomposition is $P = UV^\top$.  If $A=A^\top$ then the polar is equivalent to the matrix sign function (i.e., taking the sign of the eigenvalues). We'll use both the symmetric and rectangular polar.
Let $\Sigma = \begin{bmatrix} \Sigma_1 & 0 \\ 0 & \Sigma_2\end{bmatrix}$ where $\Sigma_1 \succeq \tau I \succ \Sigma_2$, and let $U=[U_1,U_2]$ and $V=[V_1,V_2]$ according to the same partition.

Our goal is to compute singular value soft-thresholding,
\begin{equation}
    A_\tau \defeq U \lfloor \Sigma - \tau I\rfloor_{+} V^\top = U_1 (\Sigma_1 - \tau I) V_1^\top
\end{equation}
since it is a common subroutine used in proximal gradient methods involving the matrix nuclear norm; see, e.g., \cite{cai2010singular}.  Our main result is the following proposition:

\begin{proposition}
Let $P=\polar(A)$ and $S=\polar(A^\top A - \tau I_{n})$, then $A_\tau = \frac12(A+AS) - \frac{\tau}{2}(P+PS)$.\footnote{An equivalent statement holds with $S=\polar(A A^\top - \tau I_{m})$ and $A_\tau = \frac12(A+SA) - \frac{\tau}{2}(P+SP)$, which would be faster if $n>m$.}
\end{proposition}

The intuition in the proof is that the scalar thresholding function $\lfloor x \rfloor_+ \defeq \max(x,0)$ can be written  as $\lfloor x \rfloor_+ = \frac12\left( |x| + x \right)$.
\begin{proof}
This follows from the facts that $P=UV^\top = U_1V_1^\top + U_2V_2^\top$ and also $S=V_1V_1^\top - V_2V_2^\top$. Thus $AS=U_1\Sigma_1V_1^\top - U_2\Sigma_2V_2^\top$ and $PS=U_1V_1^\top - U_2V_2^\top$, and after canceling some terms we get the result.
\end{proof}

Simplifying to the $m=n$ case, it is clear that the algorithm is still $\mathcal{O}(n^3)$ flops, no better than SVD-based methods, but it has much lower communication complexity.

\begin{corollary}
If $\sigma_{r}$ is known (and assuming $\sigma_{r+1}>\sigma_r$), then the best rank-$r$ approximation $[A]_r$ can be computed using $S=\polar(A^\top A - \tau I_{n})$ for $\tau=\sigma_r + \epsilon$ and $[A]_r = \frac12(A+AS)$, where $0<\epsilon <\sigma_{r+1}-\sigma_r$.
\end{corollary}

\section{Numerical Experiments}
In the CPU, the code is much slower than running the standard SVD. However, on the GPU there is often an advantage.  We use Polar Express~\cite{amsel2025polar}, which converts single precision to \texttt{bfloat16} for even greater speed gains, at the expense of even less accuracy. This is an improved variant of Newton-Schulz and is fundamentally an iterative method, where the accuracy improves with more iterations. Thus our first test is to examine the accuracy as a function of number of iterations used inside the polar decomposition algorithm. The results in Fig.~\ref{fig:accuracy} show that, for one particular matrix, error decays quickly from about 4 to 6 iterations, and then stagnates. The error is always unacceptably large for some values of $\tau$ (e.g., close to $\sigma_{\max}$) while reasonable (1\%) for values of $\tau$ near $\sigma_{\min}$.  On a GPU, running extra iterations has very little cost, so we suggest overestimating, such as choosing $20$ iterations.

\begin{figure}
    \centering
    \includegraphics[width=0.495\linewidth]{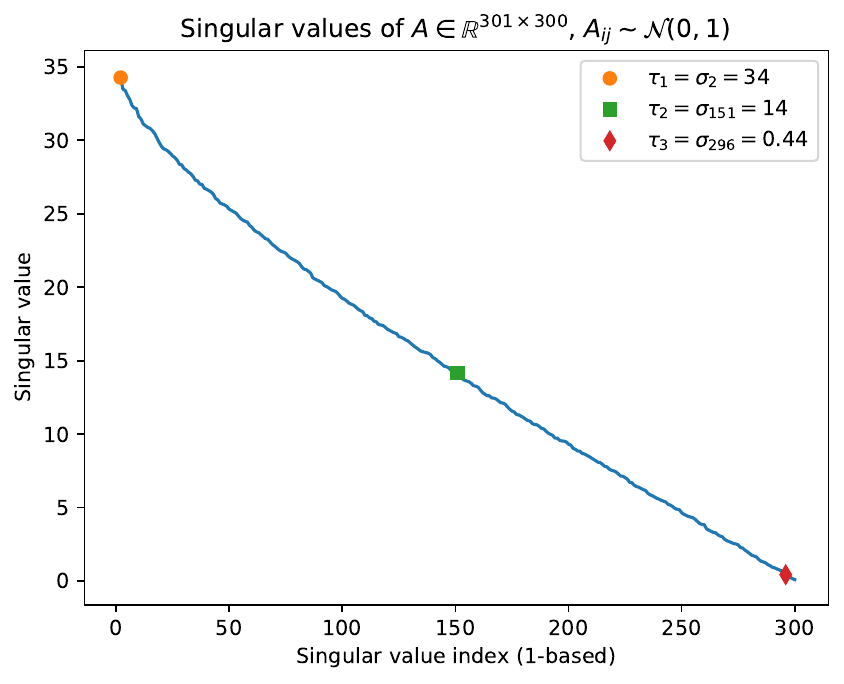}
    \includegraphics[width=0.495\linewidth]{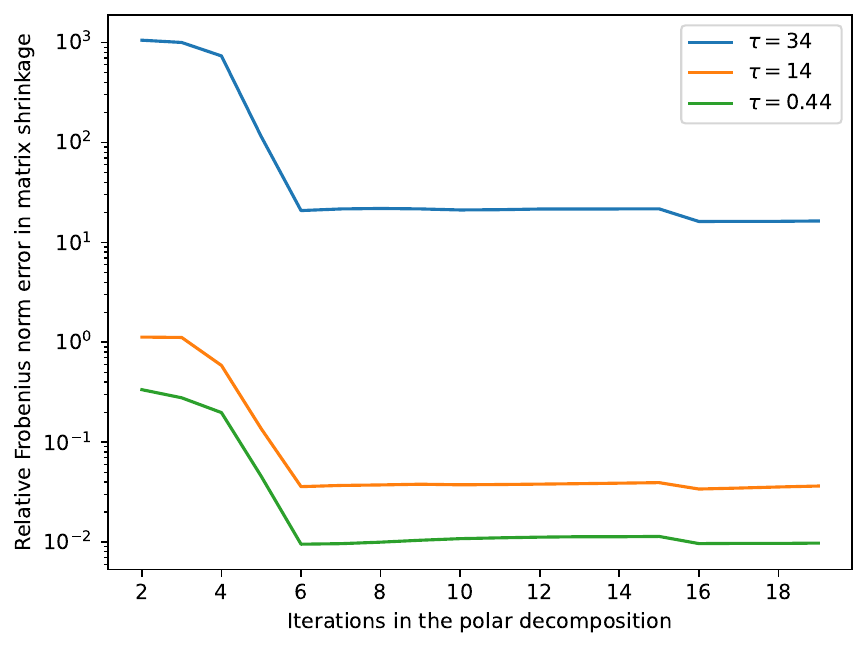}
    \caption{Accuracy test on a $301\times 300$ iid Gaussian test matrix, looking at $\|A_\tau - \widetilde{A}_\tau\|_F / \|A_\tau\|_F$ for a range of $\tau$ values, where $\widetilde{A}_\tau$ is the approximation from Proposition 1. Left: singular values of the matrix, and three test values of $\tau$. Right: relative accuracy as a function of number of iterations used in the polar decomposition subroutine.}
    \label{fig:accuracy}
\end{figure}

Next, Fig.~\ref{fig:speed} compares the wall-clock runtime of the baseline SVD-based algorithm with the polar-based algorithm, for matrices of varying size $m\times n$ for $m\in\{1000,2000,3000,4000,5000\}$ and $n=m+1$, averaged over 5 draws of a random matrix, and using $25$ iterations for the Polar Express subroutine. This was testesd on a range of GPUs available on Google Colab as of July 2026. We see that for all three GPUs, the polar-based algorithm is approximately $10\times$ faster for all matrix sizes. The scaling with $m$ is roughly the same for both algorithms, which is expected.

\begin{figure}
    \centering
    \includegraphics[width=0.5\linewidth]{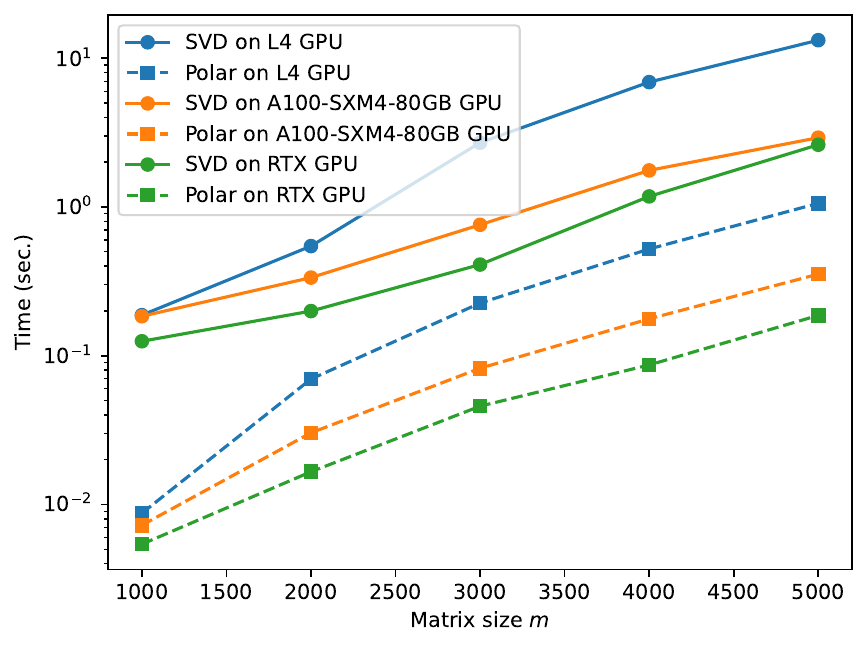}
    \caption{Comparing the speed of the baseline implementation versus the GPU-friendly polar decomposition version, on several GPUs.  The solid lines show the reference SVD-based algorithm, while the dashed lines show the polar decomposition based algorithm from Proposition 1. The different colors correspond to different GPUs. 
    For all three GPUs, the GPU-friendly version is approximately $10\times$ faster, consistently over a range of matrix sizes.}
    \label{fig:speed}
\end{figure}

\section{Code}
Our implementation is in Python using Torch. 
Code for the reference algorithm:
\begin{lstlisting}
def threshold_helper(U,S,Vh,tau):
    """ Given A=U@diag(S)@Vh, this shrinks S by tau, cutoff at zero """
    return (U * torch.maximum(S-tau,torch.tensor(0.))) @ Vh

def singular_value_threshold(A, tau):
    U, S, Vh = torch.linalg.svd(A, full_matrices=False)
    return threshold_helper(U,S,Vh,tau)
\end{lstlisting}

Code for the proposed algorithm, using Polar Express~\cite{amsel2025polar}, downloaded from \url{https://raw.githubusercontent.com/NoahAmsel/PolarExpress/refs/heads/main/polar_express.py}:
\begin{lstlisting}
from polar_express import PolarExpress
def singular_value_threshold_via_Polar(A,tau, polarAlg=None, maxIts = 11):
    # The PolarExpress algorithm works on the last two dimensions of the tensor,
    # doing everything else batch.
    # However, this code is not yet batch compatible
    #   e.g., need to update torch.eye()
    if polarAlg is None:
        polarAlg = lambda A : PolarExpress(A,maxIts).float()
        # Note that PolarExpress returns dtype torch.bfloat16
    if A.ndim != 2: raise ValueError("input tensor must be 2D")
    m, n = A.size(-2), A.size(-1) # better than A.shape since ignores batch dim.
    UV = polarAlg(A)
    if n < m:
        S = polarAlg( A.mT @ A - tau**2 * torch.eye(n,device=A.device) )
        Y = (A+A@S)/2 - tau/2*(UV+UV@S)
    else:
        S = polarAlg( A @ A.mT - tau**2 * torch.eye(m,device=A.device))
        Y = (A+S@A)/2 - tau/2*(UV+S@UV)
    return Y.float() # convert to float32 if not already
\end{lstlisting}

% \subsection*{Acknowledgments}
% \vspace{-.5em}
% The author is grateful to P.~Combettes for helpful discussions and introducing the Chen-Teboulle algorithm.

\pdfbookmark[1]{References}{refSection}
\bibliographystyle{amsalpha}

\bibliography{refs}
\end{document}